\documentclass{amsart}
\usepackage{amsmath, amssymb}
\usepackage{url}

\newtheorem{innercustomthm}{Theorem}
\newenvironment{customthm}[1]
  {\renewcommand\theinnercustomthm{#1}\innercustomthm}
  {\endinnercustomthm}
  
\newtheorem{theorem}{Theorem}[section]
\newtheorem{lemma}[theorem]{Lemma}
\newtheorem{corollary}[theorem]{Corollary}
\newtheorem*{corollary*}{Corollary}

\newtheorem{proposition}[theorem]{Proposition}
\newtheorem{claim}[theorem]{Claim}
\newtheorem{conjecture}[theorem]{Conjecture}

\theoremstyle{definition}

\newtheorem{remark}[theorem]{Remark}
\newtheorem*{remark*}{Remark}
\newtheorem{definition}[theorem]{Definition}
\newtheorem*{definition*}{Definition}
\newtheorem{assumption}[theorem]{Assumption}

\def \U {\mathcal U}
\def \C {\mathcal C}
\def \PP {\mathbb P}

\def\GL{\operatorname{GL}}
\def\sym{\operatorname{Sym}}
\def\GR{\operatorname{GR}}
\def\PGL{\operatorname{PGL}}
\def\DCF{\operatorname{DCF}}
\def\ACF{\operatorname{ACF}}
\def\ccm{\operatorname{CCM}}

\def\dcl{\operatorname{dcl}}
\def\acl{\operatorname{acl}}
\def\alg{\operatorname{alg}}

\def\id{\operatorname{id}}

\def\tp{\operatorname{tp}}

\def\cb{\operatorname{Cb}}

\def\RM{\operatorname{RM}}
\def\aut{\operatorname{Aut}}
\def\gtd{\operatorname{gtd}}
\def\C{\mathcal C}
\def\stab{\operatorname{Stab}}

\def\nmdeg{\operatorname{nmdeg}}
\def\Ind#1#2{#1\setbox0=\hbox{$#1x$}\kern\wd0\hbox to 0pt{\hss$#1\mid$\hss}
\lower.9\ht0\hbox to 0pt{\hss$#1\smile$\hss}\kern\wd0}
\def\ind{\mathop{\mathpalette\Ind{}}}
\def\Notind#1#2{#1\setbox0=\hbox{$#1x$}\kern\wd0\hbox to 0pt{\mathchardef
\nn=12854\hss$#1\nn$\kern1.4\wd0\hss}\hbox to
0pt{\hss$#1\mid$\hss}\lower.9\ht0 \hbox to
0pt{\hss$#1\smile$\hss}\kern\wd0}
\def\nind{\mathop{\mathpalette\Notind{}}}

\newcommand{\m}{\mathbb }
\newcommand{\mc}{\mathcal }

\title[Bounding nonminimality]{Bounding nonminimality and\\ a conjecture of  Borovik-Cherlin}

\author{James Freitag}
\address{James Freitag\\
University of Illinois Chicago\\ 
Department of Mathematics, Statistics,
and Computer Science\\ 
851 S. Morgan Street\\
Chicago, IL, 60607-7045\\
USA}
\email{jfreitag@uic.edu}

\author{Rahim Moosa}
\address{Rahim Moosa\\
University of Waterloo\\
Department of Pure Mathematics\\
200 University Avenue West\\
Waterloo, Ontario \  N2L 3G1\\
Canada}
\email{rmoosa@uwaterloo.ca}

\subjclass[2020]{03C45, 14L30, 12H05, 32J99}

\thanks{The first author was partially supported by NSF grant DMS-1700095 and NSF CAREER award 1945251. The second author was partially supported by an NSERC Discovery Grant.}

\date{\today}

\begin{document}

\begin{abstract}
Motivated by the search for methods to establish strong minimality of certain low order algebraic differential equations, a measure of how far a finite rank stationary type is from being minimal is introduced and studied:
The {\em degree of nonminimality} is the minimum number of realisations of the type required to witness a nonalgebraic forking extension.
Conditional on the truth of a conjecture of Borovik and Cherlin on the generic multiple-transitivity of homogeneous spaces definable in the stable theory being considered, it is shown that the nonminimality degree is bounded by the $U$-rank plus $2$.
The Borovik-Cherlin conjecture itself is verified for algebraic and meromorphic group actions, and a bound of $U$-rank plus $1$ is then deduced unconditionally for differentially closed fields and compact complex manifolds.
An application is given regarding transcendence of solutions to algebraic differential equations.
\end{abstract}

\maketitle

\setcounter{tocdepth}{1}
\tableofcontents

\bigskip
\section{Introduction}

\noindent
Strongly minimal sets -- namely definable sets all of whose definable subsets are either finite or cofinite -- play a central role in model theory and its applications.
Often significant effort is put into proving that certain specific definable sets of interest are strongly minimal.
This is especially true in the model theory of differentially closed fields of characteristic zero ($\DCF_0$), where showing strong minimality of certain algebraic differential equations has solved several long-standing transcendence problems.
Recent examples include:
\begin{itemize}
\item
{\em The generic Painlev\'e equations.}
As pointed out by Nagloo and Pillay~\cite{nagloo2014algebraic}, strong minimality comes out of the work of the Japanese school (Okamoto, Nishioka, Noumi, Umemura, Watanabe) in the 1980s and '90s.
\item
{\em Generic order $2$ equations over the constants.}
Strong minimality has been recently established by Jaoui~\cite{jaoui2019generic}.
\item
{\em The Schwarzian equation of the $j$-function.}
The first author, first with Scanlon in~\cite{freitag2017strong}, used Pila's Ax-Lindemann-Weierstrass theorem to prove strong minimality, and then with Casale and Nagloo in~\cite{cfn}, gave a new proof that generalises to the uniformisers of certain Fuchsian groups.
\end{itemize}
The proofs of strong minimality in the above cases vary significantly and are quite specific, not just to the theory of differentially closed fields, but to the particular equation being studied.
They are also difficult. 

This paper is motivated by the search for general techniques that might aid in establishing strong minimality.
While we remain primarily interested in algebraic differential equations, our basic approach is abstract and stability-theoretic.
Recall that a complete type is {\em minimal} if it has no nonalgebraic forking extensions.
It is well known that, in a stable theory, if $p$ is not minimal then a nonalgebraic forking extension can be found over a finite set of realisations of $p$ itself.
We thus propose the following naive measure of nonminimality:

\begin{definition*}[Degree of nonminimality]
Suppose $p\in S(A)$ is a nonalgebraic and nonminimal stationary type.
By the {\em degree of nonminimality} of $p$, denoted by $\nmdeg(p)$, we mean the least~$k$ such that $p$ has a nonalgebraic forking extension over $A\cup\{a_1,\dots,a_k\}$, for some $a_1,\dots,a_k$ realising $p$.
\end{definition*}

We wish to bound the degree of nonminimality in terms of $U$-rank.
If the bound is good, and the $U$-rank is small\footnote{For example, the order of the algebraic differential equation bounds the $U$-rank in $\DCF_0$, and the examples given at the beginning of the Introduction were all of order two or three.}, then this can significantly limit the parameter spaces one needs to consider when proving minimality.

In some cases we succeed unconditionally:

\begin{customthm}{A}
\label{A}
In $\DCF_0$, and in the theory $\ccm$ of compact complex manifolds, every complete stationary type of finite rank satisfies $\nmdeg(p)\leq U(p)+1$.
\end{customthm}

This appears in Theorems~\ref{bound-dcf} and~\ref{bound-ccm}, below.

Here is an application of Theorem~\ref{A} to questions of transcendence for solutions to algebraic differential equations.
Given an algebraic differential equation of order~$n$ over a differential field $k$, consider the following condition:
\begin{itemize}
\item[$(C_m)$]
For any $m$ distinct solutions $a_1,\dots,a_m\notin k^{\alg}$, the sequence
$$(\delta^{(i)}a_j:i=0,\dots,n-1,j=1,\dots,m)$$
is algebraically independent over $k$.
\end{itemize}
Note that $(C_m)$ increases in strength as $m$ increases.
This condition has been studied for many particular classes of equations.
For example, a consequence of the work in~\cite{nagloo2014algebraic} mentioned above is that the generic Painlev\'e equations satisfy $(C_m)$ for all $m$.
It is shown in~\cite{jaoui2019generic} that for generic order~2 equations over constant parameters, $(C_2)\implies(C_m)$ for all $m$.
The same holds for the (order~3) Schwarzian equations corresponding to the Fuchsian groups studied in~\cite{cfn}, and in that case it is understood when $(C_2)$ holds -- for example, it follows from~\cite{freitag2017strong} that $(C_2)$ fails of the equation satisfied by the $j$-function.

Our Theorem~\ref{A} implies:

\begin{corollary*}
Suppose $k$ is a differential field, $n>1$, and
\begin{equation}
\label{ode}
P(x,\delta x,\dots,\delta^nx)=0
\end{equation}
is an order $n$ algebraic differential equation where $P\in k[X_0,\dots,X_n]$ is irreducible.
Then $(C_{n+2})\implies(C_m)$ for all $m$.
\end{corollary*}

\begin{proof}
We work in $\DCF_0$, denote by $V$ the Kolchin closed set defined by ~(\ref{ode}), and assume that $(C_{n+2})$ holds of $V$.
We argue that $V$ is a trivial strongly minimal set.
This will suffice as trivial strongly minimal sets, more or less by definition, satisfy $(C_2)\implies(C_m)$ for all $m$.

In particular, $(C_1)$ holds of $V$.
This already implies that $V$ has no infinite Kolchin closed subsets over $k$ of order less than $n$.
Together with irreducibility of $P$, this implies that $V$ is Kolchin irreducible\footnote{Irreducibility follows, for example, from the structure of prime differential ideals in one variable explained on pages 39--41 of~\cite{mmp}} and that all the nonalgebraic points realise the Kolchin generic type $p\in S(k)$.
We claim that $p$ is minimal.
Indeed, if not, then let $d:=\nmdeg(p)$.
There are realisations $a_1,\dots,a_d\models p$ and a nonalgebraic forking extension $q\in S(ka_1,\dots,a_d)$ of $p$.
Letting $a_{d+1}$ be a realisation of $q$, we have that $(a_1,\dots,a_{d+1})$ witnesses the failure of $(C_{d+1})$.
By Theorem~\ref{A}, $d\leq U(p)+1\leq n+1$.
But this contradicts the assumption that $(C_{n+2})$ holds.

As every nonalgebraic point in $V$ realises the minimal type $p$, we conclude that~$V$ itself is a strongly minimal set.
It now follows from well known facts, together with $(C_4)$, that the geometry of $V$ is trivial.
First of all, since $n>1$, $V$ must be (locally) modular.
If it is nontrivial then by work of Hrushovski and Sokolovic, it is nonorthogonal to a modular group $G$ defined over $k^{\alg}$.
(See, for example Fact~4.1 of~\cite{isolated} and the discussion following.)
In fact, there is a generically finite-to-finite correspondence between $V$ and $G$ over $\acl(ka_1)$, where $a_1\in V$.
(We are using here that nonorthogonal modular minimal sets are not weakly orthogonal, see~\cite[Corollary~2.5.5]{GST}).
The group structure on $G$ gives distinct $a_2,a_3,a_4\in V$ such that $a_4\in\acl(k,a_1,a_2,a_3)$.
So $(a_1,a_2,a_3,a_4)$ witnesses that $(C_4)$ fails.
This is a contradiction because $n+2\geq 4$ and we are assuming that $(C_{n+2})$ holds.
\end{proof}

\begin{remark*}
The assumption that the order be greater than~$1$ is necessary.
For example, it is shown in~\cite[Proposition 2.2]{nagloo2020algebraic}, and also discussed in~$\S$\ref{riccati} below, that the Riccati equation $\delta x=ax^2 + bx+c$ over $\big(\mathbb C(t),\frac{d}{dt}\big)$, if it has no solutions in $\mathbb C(t)^{\alg}$, satisfies~$(C_3)$.
But it can be shown that $(C_4)$ does not hold.
On the other hand, even for order~1 equations, we do have that $(C_4)\implies(C_m)$ for all $m$.
Indeed, the proof of the Corollary implies that it suffices to rule out the case of a strongly minimal set that is nonmodular, and hence nonorthogonal to the constants.
This would give rise to a definable group action on a strongly minimal set, and it is not hard to show, using the classification of such actions, that then $(C_4)$ fails.
\end{remark*}

Our main motivation for Theorem~\ref{A}, however, is to aid in showing that particular low order equations are strongly minimal.
Let us be a bit more explicit.
Suppose one is trying to show that a given order~$3$ algebraic differential equation over a differential field $k$ is strongly minimal.
Essentially, what is required is to show that there are no Kolchin closed subsets of order~$1$ or~$2$.
Given the original equation, one would expect to rule out such subvarieties by algebraic and/or computational means.
The difficulty is that the potential lower order Kolchin closed subsets may not necessarily be defined over $k$, but rather over some differential field $F\supseteq k$.
What Theorem~\ref{A} does, effectively, is allow one to restrict attention to $F=k\langle a_1,a_2,a_3,a_4\rangle$, where the $a_i$ are independent solutions to the original equation.
With $F$ now fixed, the complexity of the problem is reduced, and it may in certain cases become vulnerable to algebraic and computational approaches.
Indeed, two particular applications that we have in mind, and that are already the subject of ongoing investigation, are the following:
\begin{itemize}
\item
A more elementary proof of strong minimality of the Schwarzian equation for the $j$-function than was given in~\cite{freitag2017strong} or~\cite{cfn}, and that applies to more general Schwarzian-type equations. 
\item
Strong minimality of generic algebraic differential equations of arbitrary order and sufficiently high degree, over possibly nonconstant parameters.
\end{itemize}

\medskip
Our proof of Theorem~\ref{A} goes via a conjecture of Borovik and Cherlin~\cite{BC2008}.
This conjecture is discussed in~$\S$\ref{section-bound}, but, briefly put, it asserts that if $(G,S)$ is a definable faithful and transitive group action (i.e., a {\em definable homogeneous space}) of finite Morley rank such that $G$ has a generic orbit in $S^{n+2}$, where $n=\RM(S)$, then $(G,S)$ isomorphic to the action of $\PGL_{n+1}(F)$ on $\mathbb P^n(F)$ for some algebraically closed field $F$.
In the general setting we obtain the following conditional result:

\begin{customthm}{B}
\label{B}
Suppose $T$ is a totally transcendental theory in which every nonmodular minimal type is nonorthogonal to a minimal type-definable set over the empty set that eliminates imaginaries.
If the Borovik-Cherlin conjecture holds in~$T$ then every complete stationary type of finite rank satisfies
$\nmdeg(p)\leq U(p)+2$.
\end{customthm}

This is part of Theorem~\ref{bound}, below.
In fact, there is a coarser form of the Borovik-Cherlin conjecture (stated as~\ref{BCboundConj} below) which we show is enough to give the slightly weaker bound of $\nmdeg(p)\leq U(p)+3$.

As the Borovik-Cherlin conjecture is known in Morley rank $2$ by~\cite{altinel2018recognizing}, we are able to conclude, unconditionally, that for totally transcendental theories in which every nonmodular minimal type is nonorthogonal to a minimal type-definable set over the empty set admitting elimination of imaginaries, every $U$-rank $2$ type has degree of nonminimality at most $4$.

The connection between degree of nonminimality and the conjecture of Borovik and Cherlin arises as follows:
First, as is pointed out in Section~\ref{section-nmdeg} below, we can reduce to types that admit {\em no fibrations} in the sense of~\cite{moosa2014some}.
By observations in that paper, this leads to a further reduction to types that are internal to a nonmodular minimal type.
It is to the binding group action corresponding to this internality that the Borovik-Cherlin conjecture is applied.

The route from Theorem~\ref{B} to Theorem~\ref{A} is as follows.
Since the proof of Theorem~\ref{B} only involves binding group actions of types that are internal to a nonmodular minimal type, and since all such binding group actions in $\DCF_{0}$ (and $\ccm$) are definably isomorphic to group actions definable in algebraically closed fields of characteristic zero ($\ACF_0$), to derive the bound of $U$-rank plus $2$ in $\DCF_{0}$ and $\ccm$ one only requires the Borovik-Cherlin conjecture to hold in $\ACF_0$.
That is, one requires the Borovik-Cherlin conjecture to hold of algebraic homogeneous spaces in characteristic zero.
A strategy for proving this using the work of Popov~\cite{Popov2007} was proposed in~\cite{BC2008}.
We implement it here, and in fact obtain at once the same result for meromorphic group actions as well.
Namely:

\begin{customthm}{C}
\label{C}
The Borovik-Cherlin conjecture holds of homogeneous spaces definable in $\ACF_0$ and in $\ccm$.
\end{customthm}

This appears in Theorems~\ref{BCacf} and~\ref{BCccm}.

So, combining Theorems~\ref{B} and~\ref{C}, we deduce Theorem~\ref{A} but with the weaker bound of $U(p)+2$.
Additional arguments are then required to bring that down to $U(p)+1$.
As the reader may have guessed from the above outline, the only thing special about $\DCF_0$ and $\ccm$ is that they satisfies the Zilber trichotomy: every nonmodular minimal type is nonorthogonal to a stably embedded pure algebraically closed field of characteristic zero.
So Theorem~A will hold for any such theory.
In particular, it holds in the partial differential case of $\DCF_{0,m}$ where there are $m$ commuting derivations.

Concerning lower bounds on the degree of nonminimality, the situation is highly unsatisfactory.
Examples of degree $>1$ seem very difficult to produce.
In fact, we were unable construct any of degree strictly greater than $2$.

We now describe the organisation of the paper.
In Section~\ref{section-nmdeg} we give some first properties of our degree of nonminimality, and in particular relate it to the notion of admitting no fibrations.
Then, in Section~\ref{examples}, we explore some examples of nonminimality degree in $\DCF_0$, giving an example where that degree is $2$.
We also take the opportunity in~$\S$\ref{riccati} to give, as a brief aside but using similar methods, a new proof of the above mentioned fact that the Riccati equation satisfies~$(C_3)$. 
In Section~\ref{section-bound} we explain how the conjectures of Borovik and Cherlin lead to bounds on the degree of nonminimality.
In Sections~\ref{section-bcalg} we prove the Borovik-Cherlin conjecture in $\ACF_0$.
In Section~\ref{section-dcf} we deduce the desired bound on nonminimality degree in $\DCF_0$.
We conclude, in the final section, by considering $\ccm$ and showing that both the desired bounds and the Borovik-Cherlin conjecture hold there.

\bigskip
\section{Some preliminaries on binding groups}

\noindent
Largely to fix notation and terminology, we recall various facts about internality.
Let $T$ be a complete stable theory eliminating imaginaries, and work in a sufficiently saturated $\mathcal U\models T$.

Given a complete stationary type $p\in S(A)$ and a partial type $r$ over $A$, recall that $p$ is said to be {\em $r$-internal} if there is some parameter set $B\supseteq A$ such that $p(\U)\subseteq\dcl\big(B,r(\U)\big)$.
When this is the case, it is natural to consider the group, denoted by $\aut_A(p/r)$, of permutations of $p(\U)$ which extend to elements of $\aut_A(\U)$ that fix the realisations of $r$ pointwise.
An important theorem in geometric stability theory is that $\aut_A(p/r)$, along with its action on $p(\U)$, is type-definable.
Or, to be more precise, there is a type-definable group $G$ over $A$ acting relatively $A$-definably on $S:=p(\U)$, and an (abstract) isomorphism of groups, $\aut_A(p/r)\to G$, that preserves the action of both on $S$.
We often just identify $\big(\aut_A(p/r),S)$ with it's type-definable manifestation $(G,S)$, and refer to it as the {\em binding group action of $p$ over $r$}.
Of course, when $T$ is totally transcendental (which is the situation we are mainly interested in), the group $G$ is outright definable and not just type-definable.

Recall that $p$ is said to be {\em weakly orthogonal} to $r(\U)$ if whenever $a\models p$ and $\overline b$ is a tuple of realisations of $r$ then $a\ind_A\overline b$.
We will use the fact that if $p$ is not weakly orthogonal to $r(\U)$ then it admits a nonalgebraic function to $r(\U)$. That is,

\begin{lemma}
\label{function}
Suppose $r$ is a partial type over $A$ and the induced structure on $r(\U)$ admits elimination of imaginaries.
If $p=\tp(a/A)$ is not weakly orthogonal to $r(\U)$ then there is $b\in\dcl(Aa)\setminus\acl(A)$ such that $b\models r$.
\end{lemma}

\begin{proof}
It is a well known consequence of stable embeddability that
$$\tp\big(a/\dcl(Aa)\cap\dcl(Ar(\U)\big)$$
isolates $\tp(a/Ar(\U))$. See, for example, Lemma~1 of the Appendix to~\cite{Chatzidakis99modeltheory}.
That $p$ is not weakly orthogonal to $r(\U)$ therefore implies the existence of $\overline b\in \dcl(Ar(\U))$ such that $\overline b\in\dcl(Aa)\setminus\acl(A)$.
That $r$ admits elimination of imaganaries tells us that $\overline b$ is interdefinable over $A$ with a tuple of realisations of $r$.
At least one of the co-ordinates of that tuple will be outside $\acl(A)$, and setting $b$ to be such a co-ordinate works.
\end{proof}

It is also well known, and easy to check using stationarity, that if $p$ is $r$-internal but weakly $r(\U)$-orthogonal, then the binding group action is transitive.
If in addition $T$ is totally transcendental, then it follows that $S$ is also definable, and so $p$ is isolated.
A useful consequence is that the connected component acts transitively:

\begin{lemma}
\label{connectedbind}
Suppose $T$ is totally transcendental, $r$ is a partial type over $A$, and $p\in S(A)$ is stationary, $r$-internal, and weakly $r(\U)$-orthogonal.
Then the connected component of $\aut_A(p/r)$ also acts transitively on $p(\U)$.
\end{lemma}

\begin{proof}
Let $G:=\aut_A(p/r)$ and $S:=p(\U)$.
As $G^\circ$ is of finite index in $G$, and $G$ acts transitively on $S=p(\U)$, we have that $S$ is a finite union of $G^\circ$-orbits.
By stationarity, $S$ is of Morley degree $1$, and so exactly one of these orbits, say $\mathcal O$, is of Morley rank $\RM(S)$.
But then $\mathcal O$ is $A$-invariant, and hence $A$-definable.
Since $p$ is isolated, this forces $\mathcal O=S$.
\end{proof}

Another bit of notation that will be useful is $p^{(k)}$.
This is the type of the Morley sequence in $p$ of length $k$, that is, the type of an independent sequence of $k$ realisations.
Note that via the diagonal action, $\aut_A(p/r)=\aut_A(p^{(k)}/r)$.

\bigskip
\section{Degree of nonminimality and fibrations}
\label{section-nmdeg}
\noindent
Work in a sufficiently saturated model $\mathcal U$ of a complete stable theory $T$ admitting elimination of imaginaries.

Recall from the introduction that we defined the {\em degree of nonminimality} of a stationary type $p\in S(A)$  of $U$-rank greater than $1$, denoted by $\nmdeg(p)$, to be the least~$k$ such that $p$ has a nonalgebraic forking extension over $Aa_1,\dots,a_k$, for some $a_1,\dots,a_k$ realising $p$.
If $U(p)\leq 1$ then we set $\nmdeg(p)=0$.

\begin{lemma}
\label{basicprop}
\begin{itemize}
\item[(a)]
The nonminimality degree exists.
\item[(b)]
Any witness to $\nmdeg(p)$ is a Morley sequence in $p$.
\item[(c)]
Degree of nonminimality is preserved under interalgebraicity.
\end{itemize}
\end{lemma}

\begin{proof}
Part~(a).
Note that $U(p)>1$ ensures that a (stationary) nonalgebraic forking extension exists, say $q\in S(B)$.
Let $e=\cb(q)$.
Then $e$ is in the definable closure of a finite Morley sequence in $q$, say $(a_1,\dots,a_k)$.
Choose $b\models q$ with $b\ind_{B}a_1,\dots, a_k$.
Then $b\ind_{Ae}Ba_1,\dots, a_k$.
Since $q|_{Ae}$ is a nonalgebraic forking extension of $p$, we have that $\tp(b/Aa_1,\dots,a_k)$ is also.

Part~(b).
Suppose $a_1,\dots,a_k$ witnesses that $\nmdeg(p)=k$.
If $a_i\nind_A(a_1,\dots,a_{i-1})$ then the forking extension $\tp(a_i/Aa_1,\dots,a_{i-1})$ is either algebraic, in which case we could drop $a_i$ from $(a_1,\dots,a_k)$ and contradict minimality of $k$, or $(a_1,\dots,a_{i-1})$ would witness that the degree of nonminimality is $<k$, which is also impossible.

Part(c).
Suppose $p'\in S(A)$ is interalgebraic with $p$ over $A$.
Suppose $a_1,\dots,a_k,b$ are realisations of $p$ such that $\tp(b/Aa_1,\dots, a_k)$ is a nonalgebraic forking extension of $p$.
Then there are $a_1',\dots,a_k',b'$ realising of $p'$ such that $\acl(Ab)=\acl(Ab')$ and $\acl(Aa_i)=\acl(Aa_i')$ for all $i=1,\dots,k$.
Hence $\tp(b'/Aa_1',\dots, a_k')$ is a nonalgebraic forking extension of $p'$.
\end{proof}

Note that, as a consequence of part~(b) for example, one does not expect the degree of nonminimality to be preserved under taking nonforking extensions.
Indeed, if $p$ is nonminimal with $\nmdeg(p)=k$ witnessed by a Morley sequence $a_1,\dots,a_k$, then $\nmdeg(a_k/Aa_1\cdots a_{k-1})=1$.

As the following lemma illustrates, the degree of nonminimality is frequently~$1$.

\begin{lemma}
\label{fibration=1}
Suppose $p=\tp(a/A)$ and there exists $b\in\acl(Aa)\setminus\acl(A)$ such that $a\notin\acl(Ab)$.
Then $\nmdeg(p)=1$.
\end{lemma}

\begin{proof}
Note that $a\nind_Ab$.
Let $a'$ realise the nonforking extension of $\tp(a/Ab)$ to $Aba$.
Then $a'\nind_Ab$ also, and so $a'\nind_Aa$.
If $a'\in\acl(Aa)$ then, as $a'\ind_{Ab}a$, we would have $a'\in\acl(Ab)$ which would imply that $a\in\acl(Ab)$, which is not the case.
So $\tp(a'/Aa)$ is a nonalgebraic forking extension of $p$.
\end{proof}

In~\cite{moosa2014some} the notion of a {\em proper fibration} of a type $p(x)=\tp(a/A)$ was introduced: it is a nonalgebraic type $\tp(b/A)$ where $b\in \dcl(Aa)$ and $a\notin\acl(Ab)$.
Lemma~\ref{fibration=1} says that if $p$ admits a finite cover which admits a proper fibration then $\nmdeg(p)=1$.
One consequence of this is that if $\nmdeg(p)>1$ then $p$ is algebraic over a $1$-type; thus reducing the search for types of high nonminimality degree to $1$-types.
In particular, in a strongly minimal theory all types are either algebraic, minimal, or have degree of nonminimality $1$.

Another consequence of $\nmdeg>1$, following from some observations in~\cite{moosa2014some}, is internality to a nonmodular minimal type:

\begin{proposition}
\label{internality}
Suppose $p\in S(A)$ is of finite $U$-rank.
If $\nmdeg(p)>1$ then there is a non locally modular minimal type $r$ such that $p$ is almost $r$-internal.
\end{proposition}

\begin{proof}
By Lemma~\ref{fibration=1}, $p$ does not admit any proper fibration.
Proposition~2.3 of~\cite{moosa2014some} tells us that any finite $U$-rank type that admits no fibrations is semiminimal.
In fact it says more, that either $p$ is interalgebraic with $q^{(k)}$ for some modular minimal type $q\in S(A)$, or $p$ is almost $r$-internal for some non locally modular minimal type~$r$ over possibly additional parameters.
In the former case, $k>1$ since $U(p)>1$, and by Lemma~\ref{basicprop}(c) we have $\nmdeg(p)=\nmdeg(q^{(k)})$.
But as $q^{(k)}$ clearly admits a proper fibration -- namely $q$ itself via any co-ordinate projection -- and hence has degree of nonminimality $1$, this is impossible.
Hence, $p$ is almost $r$-internal.
\end{proof}

\bigskip
\bigskip
\section{Some examples}
\label{examples}
\noindent
In this section we work in $\U\models\DCF_0$, a sufficiently saturated differentially closed field of characteristic zero, with field of constants $\C$.

\subsection{Degree of nonminimality $1$ but without proper fibrations}
The converse of Lemma~\ref{fibration=1} is not true: there are types of $\nmdeg 1$ that do not admit proper fibrations, nor do any of their finite covers.

\label{abelian}
Consider the type $p$ of Example~5.4 of~\cite{moosa2014some}.
Namely, let $A$ be a simple abelian variety of dimension $d>1$ over a subfield $k\subseteq\C$.
Let $\ell:A\to T_0A=\mathbb G_a^d$ be the logarithmic derivative and set $G$ to be the subgroup of $A$ defined by
$$G:=\{g\in A:\ell(g)=(c,0,\dots,0)\text{ for some }c\in\C\}.$$
Then $G$ is an extension of $\mathbb G_a(\C)$ by $A(\C)$.
For $a\in G$ Kolchin generic over $k$, the type $p:=\tp(a/k\ell(a))$ is stationary, $\C$-internal, and of dimension~$d$.
Moreover, it is shown in~\cite[Example~5.4]{moosa2014some} that $p(\U)=a+A(\C)$.
Hence, $\nmdeg(p)=1$.
Indeed, if $X\subseteq A(\C)$ is any proper irreducible subvariety over $k$ of positive dimension, and $x\in X$ generic over $ka$, then $q:=\tp(a+x/ka)$ is  the Kolchin generic type of $a+X$ over $ka$, and hence a nonalgebraic type in $a+A(\C)$ of dimension $<d$.
So $q$ is a nonalgebraic forking extension of $p$.

It is also shown in~\cite{moosa2014some} that $p$ admits no proper fibrations, and the proof given their extends to show that $p$ admits no finite covers that have proper fibrations.
We give some details.
Let $b\in\acl(ka)$.
We wish to show that either $a\in\acl(k\ell(a)b)$ or $b\in\acl(k\ell(a))$.
Set $q:=\tp(a/k\ell(a)b)$.
The binding group of $p$ is $A(\C)$ acting by translation.
By commutativity, the binding group $H$ of $q$ is a definable subgroup of the simple abelian variety $A(\C)$.
So $H$ is either finite or all of $A(\C)$.
If $H$ is finite then $a\in\acl(k\ell(a)b\C)$.
But as the binding group of $p$ acts transitively on $p(\U)$ we must have that $a\ind_{k\ell(a)}\C$ and hence $a\ind_{k\ell(a)b}\C$ as $b\in\acl(ka)$.
It would therefore follow in this case that $a\in\acl(k\ell(a)b)$.
On the other hand, if $H=A(\C)$ then the transitivity of the action of the binding group of $p$ implies that $p\vdash q$, and so $q$ is a nonforking extension, which forces $b\in\acl(k\ell(a))$.\qed

\bigskip
\subsection{Degree of nonminimality $2$}
\label{pgl}
It seems difficult to find types of $\nmdeg>1$.
Any two realisations of such a type must be either interalgebraic or independent.
Moreover, as we have seen, such a type must be internal to a nonmodular minimal type, and by arguments similar to those in Example~\ref{abelian} above, one can deduce various restrictions on the binding group action.

Nevertheless, there are some natural examples.
For instance, a nonminimal internal type with a binding group action that is $2$-transitive would necessarily have $\nmdeg>1$.
Indeed, every pair of distinct realisations is an independent pair, so that over any one realisation there are no nonalgebraic forking extensions.

This happens in $\DCF_0$.
For each $n\geq 2$, we will exhibit a $\mathcal C$-internal type $p$ whose binding group action is isomorphic to the natural action of $\PGL_n$ on $\mathbb P^{n-1}$.
When $n\geq 3$ this action is $2$-transitive.
To construct $p$ we will follow~\cite[Example~4.2]{jin-moosa} which dealt with the case of $n=2$.

Fix $n\geq 2$.
Let $B\in M_n(\mathcal U)$ be an $n\times n$ matrix whose entries form an independent set of differential transcendentals, and consider the system of order~$1$ linear differential equations given by
\begin{equation}
\label{V}
\delta X=BX
\end{equation}
where $X$ is an $n$-column vector of variables.
Let $V$ be the set of solutions to~(\ref{V}), viewed as a $\C$-vector subspace of $\U^n$.
Note that the natural action of $\GL_n(\U)$ on~$\U^n$ induces an action of $\GL_n(\C)$ on $V$.
From the existence of a fundamental system of solutions to~(\ref{V}), namely $n$ vectors that form both a $\C$-basis for $V$ and a $\U$-basis for~$\U^n$, it is not hard to see that this action of $\GL_n(\C)$ is that of $\GL(V)$, the group of $\C$-linear isomorphisms of~$V$.

Let $F=\mathbb Q\langle B\rangle$ be the differential field generated by the entries of $B$, and $q$ the Kolchin generic type of $V$ over $F$.
So $q$ is $\C$-internal.
Our first goal is to show that the binding group action of $q$ is the action of $\GL(V)$ on $V\setminus\{0\}$.

Note that by genericity of $q$, any realisation is an algebraically independent $n$-tuple over $F$.
So if we fix, $v_1,\dots, v_n$ independent realisations of $q$, and form the $n\times n$ matrix $M$ whose columns are the $v_i$, then the $n^2$ entries of $M$ form an algebraically independent set over $F$.
In particular $M$ is invertible.
It follows that we have a map $\rho:\aut_F(q/\C)\to\GL_n(\C)$ given by $\rho(\sigma)=M^\sigma M^{-1}$.
This is because $\delta M=BM$ and $\delta M^\sigma=BM^\sigma$.

\begin{claim}
\label{qnwo}
$q^{(n)}$ is weakly orthogonal to $\C$.
\end{claim}

\begin{proof}
Since $M\models q^{(n)}$, we need to show that $F\langle M\rangle\cap\C\subseteq F$.
In fact, we will show that $F\langle M\rangle\cap\C=\mathbb Q$.
But
\begin{eqnarray*}
F\langle M\rangle\cap\C
&=&
\mathbb Q\langle B\rangle\langle M\rangle\cap\C\\
&=&
\mathbb Q\langle M\rangle\cap\C\ \ \ \text{ since $B=M^{-1}\delta M$}
\end{eqnarray*}
So it will suffice to show that the entries of $M$ form an independent set of differential transcendentals over $\mathbb Q$.
This is the case: on the one hand the differential transcendental degree of $\mathbb Q\langle M\rangle$ is at most $n^2$ because that is the number of given generators, while on the other hand it is at least~$n^2$ because $\mathbb Q\langle B\rangle\subseteq \mathbb Q\langle M\rangle$ as we have already observed.
\end{proof}

\begin{claim}
$q(\U)=V\setminus\{0\}$
\end{claim}
\begin{proof}
Claim~\ref{qnwo} implies, in particular, that $q$ itself is weakly orthogonal to $\C$.
Hence, the action of $\GL_n(\C)$ on $V$ restricts to an action on $q(\U)$.
Indeed, given $C\in\GL_n(\C)$ and $v\models q$, we have that $v$ and $C$ are independent over $F$, and so $Cv$ is again a realisation of $q$.
Recalling that $\GL_n(\C)$ acts as $\GL(V)$ on $V$, it follows that $q(\U)=V\setminus\{0\}$.
\end{proof}

\begin{claim}
$\rho:\aut_F(q/\C)\to\GL_n(\C)$ is an isomorphism that preserves the natural action of both groups on $q(\U)$.
\end{claim}

\begin{proof}
It is easily checked that $\rho$ is a homomorphism.
For injectivity note that if $\rho(\sigma)=1$ then $M^\sigma=M$ and hence $\sigma$ fixes $v_1,\dots, v_n$.
But from Claim~\ref{qnwo} it follows that $v_1,\dots,v_n$ is a $\C$-basis for $V$, so that we must have $\sigma=\id$.
Surjectivity now also follows from Claim~\ref{qnwo} as it implies that the dimension of $\aut_F(q/\C)$ (as an algebraic group) must be $\dim(q^{(n)})=n^2=\dim\GL_n(\C)$.
Finally, fixing $\sigma\in\aut_F(q/\C)$, note that, for each $i=1,\dots,n$, $\sigma(v_i)$ is the $i$th column of $M^\sigma=\rho(\sigma)M$, and hence $\sigma(v_i)=\rho(\sigma)v_i$.
Again using the fact that  $v_1,\dots,v_n$ is a $\C$-basis for $V$, we get that $\sigma$ and $\rho(\sigma)$ agree on all of $q(\U)$.
\end{proof}

We have shown that the binding group action of $q$ is that of $\GL(V)$ on $V\setminus \{0\}$.
Our next step is to projectivise.
That is, consider $\pi:\U^n\setminus\{u_n=0\}\to \U^{n-1}$ given by $\pi(u_1,\dots,u_n)=(\frac{u_1}{u_n},\dots,\frac{u_{n-1}}{u_n})$.

\begin{claim}
\label{proj}
The restricton of $\pi$ to $V\setminus \{0\}$ is the natural projectivisation map $V\setminus\{0\}\to\mathbb P(V)$.
\end{claim}

\begin{proof}
First, note that $\pi$ is defined on $V\setminus\{0\}$ because the latter is $q(\U)$ and realisations of $q$ are algebraically independent tuples over $F$.
It is clear that if $c\in\C^*$ and $v\in V$ is nonzero then $\pi(v)=\pi(cv)$.
Conversely, suppose $\pi(v)=\pi(v')$ for some pair of nonzero $v,v'\in V$.
Then $v'=av$ for some $a\in \U^*$.
On the one hand we have $\delta(av)=(\delta a)v+aBv$.
On the other hand
$\delta(av)=B(av)=aBv$.
It follows that $\delta a=0$, and $a\in\C^*$ as desired.
\end{proof}

Let $p:=\pi(q)\in S_{n-1}(F)$.
This is a $U$-rank $n-1$ type that is $\C$-internal.
By Claim~\ref{proj}, $p(\U)=\mathbb P(V)$.
Moreover, we have an induced surjective homomorphism $\widehat\pi:\aut_F(q/\C)\to\aut_F(p/\C)$ that is compatible with the actions of these groups on $q(\U)$ and $p(\U)$ respectively, see for example~\cite[Lemma~3.1]{jin-moosa}.
Under the identification of $\aut_F(q/\C)$ with $\GL(V)$, we see by Claim~\ref{proj}, that the kernel of $\widehat\pi$ is $\C^*$.
Indeed, the elements of $\ker(\widehat\pi)$ are precisely those elements of $\GL(V)$ that preserve every $1$-dimensional subspace of $V$, and that is clearly $\C^*$.
Hence $\aut_F(p/\C)$ acts on $p(\U)=\mathbb P(V)$ as $\GL(V)/\C^*=\PGL(V)$ does.
Since $V$ is an $n$-dimensional vector space (over $\C$) this is what we were looking for: a binding group action that is isomorphic to the action of $\PGL_n$ on $\mathbb P^{n-1}$.

In particular, if $n\geq 3$ then $p$ is nonminimal, $\C$-internal and with $\aut_F(p/\C)$ acting $2$-transitively on~$p(\U)$.
Hence $\nmdeg(p)\geq 2$.
In fact,
\begin{claim}
If $n\geq 3$ then $\nmdeg(p)=2$.
\end{claim}

\begin{proof}
We have already seen that $\nmdeg(p)\geq 2$.
To show that $\nmdeg\leq 2$ we fix a pair of independent realisations $w_1,w_2$, of $p$, and show that $p$ admits a nonalgebraic forking extension to $K:=F\langle w_1,w_2\rangle$.
This should be the generic type of the line in $\mathbb P(V)$ connecting $w_1$ and $w_2$.
That is, let $W\subseteq V$ be the $K$-definable $2$-dimensional $\C$-subspace that contains $\pi^{-1}(w_1)$ and $\pi^{-1}(w_2)$.
Since $n\geq 3$, $W$ is a proper subspace of $V$.
Let $w$ be Kolchin generic in $W$ over~$K$.
We claim that $p'=tp(\pi(w)/K)$ is a nonalgebraic forking extension of $p$.
Note that $w$ is Kolchin generic in $\pi^{-1}(\pi(w))$ over $K\langle\pi(w)\rangle$.
So, if $p'$ were a nonforking extension of $p$ then $w$ would be Kolchin generic in $V$ over $K$, contradicting the fact that it is contained in the proper subspace $W$.
Similarly, if $p'$ were an algebraic extension of $p$ then $U(w/K)$ would be $1$, contradicting the fact that $\dim_{\C}W=2$.
\end{proof}

\bigskip
\subsection{An aside on the order~$1$ Riccati equation}
\label{riccati}
In the Introduction we mentioned that if an order~$1$ Riccati equation $\delta y=ay^2 + by+c$ over $\big(\mathbb C(t),\frac{d}{dt}\big)$ has no algebraic solutions then it satisfies $(C_3)$; namely, any three distinct solutions are algebraically independent over $\mathbb C(t)$.
This is a theorem of Nagloo used in the study of certain Painlev\'e equations, see~\cite[Proposition 2.2]{nagloo2020algebraic}.
But it can also be seen using the approach of the previous section, as we now explain.

First, by a standard change of variables as in~\cite[Fact~2.2]{nagloo2020algebraic}, it suffices to consider
\begin{equation}
\label{riccati-reduced}
\delta y=-y^2+c.
\end{equation}
Let $U$ denote the set of solutions.
As we are assuming no algebraic solutions, (\ref{riccati-reduced}) isolates a complete type $p\in S(F)$ where $F:=\mathbb C(t)$.
This type is $\C$-internal, and our approach is to show that the associated binding group action is that of $\PGL_2$ on $\mathbb P^1$.
This will suffice: that action is $3$-transitive and hence any three distinct realisations will be independent.

Following~\cite{nagloo2020algebraic}, we consider the associated second order linear homogeneous differential equation $\delta^2x=cx$.
The connection to~(\ref{riccati-reduced}) is that if $\delta^2x=cx$ and $y:=\frac{\delta x}{x}$ then $\delta y=-y^2+c$.
Putting $\delta^2x=cx$ in $2\times 2$ matrix form, we get
\begin{equation}
\label{order2}
\delta X=
\begin{pmatrix}
0&1\\
c&0
\end{pmatrix}
X.
\end{equation}
Let $V$ be the space of solutions to~(\ref{order2}).
It is explained in~\cite{nagloo2020algebraic} how the results of~\cite{MR839134} and the assumption that~(\ref{riccati-reduced}) has no algebraic solutions imply that the {\em differential Galois group} of~(\ref{order2}) is $\operatorname{SL}_2(\C)$.
One of the consequences of this is that, as $\operatorname{SL}_2$ acts transitively on $\mathbb A^2\setminus\{0\}$, the set $V\setminus\{0\}$ is a complete type $q$ over $F$.
In this situation (i.e., a linear homogeneous differential equation whose nonzero solutions isolate a type), the differential Galois group can be identified with the binding group $\aut_F(q/\C)$.
The upshot is that the embedding $\rho:\aut_F(q/\C)\to \GL_2(\C)$ from~$\S$\ref{pgl} now identifies $\aut_F(q/\C)$ with $\operatorname{SL}_2(\C)$.
Proceeding as in~$\S$\ref{pgl}, we observe that $(u_1,u_2)\mapsto\frac{u_2}{u_1}$ restricts to the projectivisation $\pi:V\setminus\{0\}\to\mathbb P(V)=U$.
Indeed, this is essentially what~\cite[Proposition~2.1]{nagloo2020algebraic} says.
We get an induced surjective homomorphism $\widehat\pi:\aut_F(q/\C)\to\aut_F(p/\C)$ of group actions, whose kernel is $\operatorname{SL}_2(\C)\cap\C^*=\{1,-1\}$, and we conclude that the action of $\aut_F(p/\C)$ on $U$ is that of $\operatorname{SL}_2/\{1,-1\}=\PGL_2$ on $\mathbb P(V)=\mathbb P^1$, as desired.

\bigskip
\bigskip
\section{Toward an upper bound: generic transitivity}
\label{section-bound}
\noindent
Our goal is to find upper bounds on the degree of nonminimality in terms of natural invariants.
This has the potential to be useful as it opens up an approach to proving minimality in certain cases; such a bound would limit the parameter spaces one needs to consider when ruling out the existence of nonalgebraic forking extensions.

In searching for an upper bound we may as well assume that the degree of nonminimality is greater than~$1$.
Hence, by Proposition~\ref{internality}, we can restrict our attention to types that are internal to a (nonmodular) minimal type.
The following proposition bounds the degree of nonminimality in terms of the length of a ``fundamental system of solutions" witnessing the internality.

We return to the general setting of a fixed sufficiently saturated model $\U$ of a complete stable theory admitting elimination of imaginaries.

\begin{proposition}
\label{wo}
Suppose $p\in S(A)$ is a stationary type of finite $U$-rank, $r$ is a partial type over $A$ such that $p$ is $r$-internal, and $U(p)>U(r)$.
Assume, moreover, that the induced structure on $r(\U)$ admits elimination of imaginaries.
Then
$$\nmdeg(p)\leq \min\{k:p^{(k)}\text{ is not weakly orthogonal to }r(\U)\}.$$
\end{proposition}

\begin{proof}
Since $p$ is nonalgebraic and $r$-internal there is a $k>0$ such that $p^{(k)}$ is not weakly orthogonal to $r(\U)$.
Let $k$ be minimal such, and fix $(a_1,\dots,a_k)\models p^{(k)}$.
We will show that $p$ has a nonalgebraic forking extension over $Aa_1,\dots, a_k$.

By non weak orthogonality and elimination of imaginaries, see Lemma~\ref{function}, there is $b\models r$ with $b\in\dcl(Aa_1,\dots, a_k)\setminus\acl(A)$.
We claim that $\tp(b/Aa_1,\dots, a_{k-1})$ is a proper fibration of $\tp(a_k/Aa_1,\dots, a_{k-1})$.
We already know it is a fibration, so it remains to verify that neither $\tp(b/Aa_1,\dots, a_{k-1})$ nor $\tp(a_k/Aa_1,\dots, a_{k-1}b)$ are algebraic.
The former is by minimality of $k$ as $b\in\acl(Aa_1,\dots, a_{k-1})$ would imply that $p^{(k-1)}$ is not weakly orthogonal to $r(\U)$.
On the other hand, as $U(r)<U(p)$,
$$U(a_1,\dots,a_{k-1},b/A)\leq (k-1)U(p)+U(r)<kU(p)=U(a_1,\dots,a_k/A),$$
and so $a_k\notin\acl(Aa_1,\dots, a_{k-1}b)$, as desired.

Hence $\nmdeg(a_k/Aa_1,\dots, a_{k-1})=1$ by Lemma~\ref{fibration=1}.
That is, there exists $a'\models\tp(a_k/Aa_1,\dots, a_{k-1})$ such that $a'\nind_{Aa_1,\dots, a_{k-1}}a_k$ and $a'\notin\acl(Aa_1,\dots, a_k)$.
But then $a'\nind_{A}a_1,\dots, a_k$ also, and hence $\tp(a'/ Aa_1,\dots, a_k)$ is a nonalgebraic forking extension of $p$.
So $\nmdeg(p)\leq k$.
\end{proof}

Now, internality together with weak orthogonality yields a transitive action of the binding group on the realisations of the type.
In the totally transcendental case this is a definable homogeneous space, and understanding it should (and will) be useful.
One obstacle, however, is that the minimal type $r$ produced in Proposition~\ref{internality} need not be over the same parameters as $p$, whereas that is necessary to apply Proposition~\ref{wo}. In certain theories of interest, however, such as differentially closed fields and compact complex manifolds, every nonmodular minimal type $r$ over whatever parameters is nonorthogonal to one over the empty set.
We therefore impose the following additional assumptions on our theory:

\begin{assumption}
\label{nlm}
Suppose $T$ is totally transcendental and every nonmodular minimal type is nonorthogonal to a minimal type-definable set over the empty set that eliminates imaginaries.
\end{assumption}

Suppose that $T$ is a totally transcendental theory satisfying a strong form of the  Zilber trichotomy in the sense that there is a definable pure algebraically closed field $F$ such that every nonmodular minimal type is nonorthogonal to $F$.
Then we can name the parameters of $F$ to the language and Assumption~\ref{nlm} is now satisfied.
In particular this assumption holds in $\DCF_{0,m}$ and $\ccm$.

Under this assumption, Propositions~\ref{internality} and~\ref{wo}, together, suggest that we focus on the following context: a finite rank type $p\in S(A)$ of nonminimality degree greater than~$1$ and a minimal nonmodular partial type $r$ over the empty set such that $\aut_A(p/r)$ acts transitively on $p^{(k-1)}$ for some $k\geq\nmdeg(p)$.
That is, the action of $\aut_A(p/r)$  is ``generically $(k-1)$-transitive" on $p(\U)$.
Generic transitivity is a generalisation of a notion introduced for algebraic groups by Popov~\cite{Popov2007} and later developed in the context of groups of finite Morley rank by Borovik and Cherlin~\cite{BC2008} as follows:

\begin{definition}[Generic transitivity~\cite{BC2008}]
\label{genBC}
A definable action of a group $G$ of finite Morley rank on a definable set $S$ of finite Morley rank is \emph{generically k-transitive} if the diaganol action of $G$ on $S^k$ admits an orbit $\mc O$ such that $\RM(S^k \setminus \mc O)<\RM(S^k)$.
\end{definition} 

This weakens the classical notion of $k$-transitivity: an action of $G$ on $S$ is \emph{$k$-transitive} if for any two $k$-tuples $(x_1, \ldots , x_k )$ and $(y_1, \ldots , y_k)$ with $x_i \neq x_j, y_i \neq y_j$ for $i \neq j$, there is some $g \in G$ such that $(g x_1, \ldots , g x_k ) = (y_1, \ldots , y_k).$
A high degree of transitivity is known to impose strong structural conditions on the group.
For instance, Jordan \cite{jordan1872recherches} shows that if a finite group $G$ acts $4$-transitively on a set $S$ with the pointwise stabilizer of any four distinct elements being trivial (that is, $G$ acts {\em sharply} $4$-transitively) then $G$ must be one of $S_4,$ $S_5,$ $A_6$ or the Mathieu group~$M_{11}$.
Later, Tits~\cite{tits1951groupes} generalised Jordan's theorem to infinite groups, and Hall~\cite{hall1954theorem} loosened the sharpness requirement.
These theorems imply, for instance, that there are no infinite groups which act sharply $n$-transitively for $n \geq 4$. 
Dropping the sharpness requirement, but using the classification of finite simple groups, one can show that the only finite groups with a $4$-transitive action are symmetric, alternating, and Mathieu groups \cite[page 110]{MR1721031}.
In the case of infinite groups, there are few known restrictions on multiply transitive groups in general, but higher transitivity is very restricted in various definable contexts. For instance, there are no infinite $4$-transitive group actions definable in algebraically closed fields \cite{knop1983mehrfach} or in o-minimal structures \cite{tent2000sharply}.

In the definable category, even the \emph{weaker} notion of generic transitivity imposes strong structural consequences. 
The prototypical examples of high generic transitivity are, in a (pure) algebraically closed field $F$, the following:
\begin{itemize}
\item
The natural action of $\GL_n(F)$ on $F^n$ is generically $n$-transitive where the generic orbit consists of the set of bases for the vector space.
\item
The induced action of $\PGL_{n+1}(F)$ on $\mathbb P^n(F)$ is generically $(n+2)$-transitive with the generic orbit being the set of projective bases.
(Recall that a {\em projective basis} of $\mathbb P^n$ is a set of $n+2$ points with no $n+1$ of them lying on the same hyperplane.)
\end{itemize}
Conjecturally, one cannot get any higher generic transitivity than the latter action:

\begin{conjecture}[Borovik-Cherlin~\cite{BC2008}]
\label{BCboundConj}
If $G$ is a finite Morley rank group acting definably and generically $k$-transitively on an infinite set $S$, then $k\leq\RM(S)+2$.
\end{conjecture}

In fact, Borovik and Cherlin make the following more precise conjecture:

\begin{conjecture}[Borovik-Cherlin~\cite{BC2008}]
\label{BCConj}
Let $(G,S)$ be a connected homogeneous space of finite Morley rank with $n:=\RM(S)>0$.
If $G$ acts generically $(n+2)$-transitively then $(G, S)$ is isomorphic to the natural action of $\PGL_{n+1} (F) $ on $\mathbb P ^n (F)$, for some algebraically closed field $F$.
\end{conjecture}

Let us spell out how~\ref{BCboundConj} follows from~\ref{BCConj}.
First of all, as is pointed out in the proof of~\cite[Lemma~1.9]{BC2008}, as long as $k>1$ (which we may assume), the action of the connected component $G^\circ$ on $S$ is also generically $k$-transitive.
So we may assume that $G$ is connected.
We can also assume that $G$ acts faithfully (by replacing $G$ with an appropriate quotient) and transitively (by replacing $S$ with a generic orbit).
That is, we may assume $(G,S)$ is a connected homogeneous space that is generically $k$-transitive.
Now, suppose that $k>n+2$.
The $(G,S)$ is also generically $(n+2)$-transitive, and hence, by Conjecture~\ref{BCConj}, we may assume that $(G,S)=\big(\PGL_{n+1} (F), \mathbb P ^n (F)\big)$.
But, it is a fact that the action of $\PGL_{n+1}$ on $\mathbb P^n$ is not generically $k$-transitive for $k>n+2$; indeed, it is {\em sharply generically} $(n+2)$-transitive.
So $k\leq n+2$, as desired.

Here is how generic transitivity arises in considering degree of nonminimality:

\begin{proposition}
\label{gtbound}
We work under Assumption~\ref{nlm}.
Suppose $p\in S(A)$ is a stationary finite rank type with $d:=\nmdeg(p)>1$.
Then there exists a stationary type $q\in S(A)$ interalgebraic with $p$, and a nonmodular minimal partial type $r$ over the empty set, such that $q$ is $r$-internal, and if we let $G:=\aut_A(q/r)^\circ$ be the connected component of the binding group and $S:=q(\U)$ then $(G, S)$ is a definable homogeneous space that is generically $(d-1)$-transitive.
\end{proposition}

\begin{proof}
By Proposition~\ref{internality}, $p$ is almost internal to some nonmodular minimal type.
By Assumption~\ref{nlm}, that type is nonorthogonal to a (nonmodular) minimal partial type $r$ over the empty set, such that the induced structure on $r(\U)$ eliminates imaginaries.
So $p$ is almost $r$-internal.
From almost $r$-internality we obtain $q\in S(A)$ that is interalgebraic with $p$ and outright $r$-internal (see, for example, Lemma~3.6 of~\cite{jin-moosa}).
Note that by Lemma~\ref{basicprop}(c), $\nmdeg(q)=d$ as well.

Now, the binding group $\aut_A(q/r)$ and its action on $S:=q(\U)$ is definable as $T$ is totally transcendental.
Let $k$ be least such that $q^{(k)}$ is not weakly orthogonal to $r(\U)$.
By Proposition~\ref{wo}, $k\geq d$.
In particular, $q$ is weakly orthogonal to $r(\U)$.
It follows that $\aut_A(q/r)$ acts transitively on $S$, and that $S$ is isolated.
In fact, by Lemma~\ref{connectedbind}, the connected component $G$ also acts transitively on $S$.

By minimality of $k$ we have that $q^{(k-1)}$, and hence $q^{(d-1)}$, is weakly $r(\U)$-orthogonal.
Since $\aut_A(q/r)=\aut_A(q^{(d-1)}/r)$, we have that $G$ also acts transitively on $\mc O:=q^{(d-1)}(\mathcal U)\subseteq S^{d-1}$.
As $S$ is internal to a minimal type, $U$-rank and Morley rank agree on the induced structure on $S$, and so
$$\RM(S^{d-1}\setminus\mc O)<\RM(S^{d-1}).$$
That is, $G$ acts generically $(d-1)$-transitively on $S$.\end{proof}

We obtain the following conditional bound on the degree of nonminimality:

\begin{theorem}
\label{bound}
Suppose $T$ satisfies Assumption~\ref{nlm} and let $p\in S(A)$ be a stationary type of finite rank.
\begin{itemize}
\item[(a)]
If Conjecture~\ref{BCboundConj} holds for finite rank group actions definable in $T$ then $\nmdeg(p)\leq U(p)+3$.
\item[(b)]
If Conjecture~\ref{BCConj} holds for finite rank homogeneous spaces definable in $T$ then $\nmdeg(p)\leq U(p)+2$.
\end{itemize}
\end{theorem}

\begin{proof}
We may assume that $n:=U(p)>1$ and $d:=\nmdeg(p)>1$.
By Proposition~\ref{gtbound} there is a stationary type $q\in S(A)$ that is interalgebraic with $p$, and a (nonmodular) minimal partial type $r$ over the empty set such that $q$ is $r$-internal and such that the connected component of the binding group $G:=\aut_A(q/r)^\circ$ acts definably and generically $(d-1)$-transitively on $S:=q(\U)$.
Since $\RM(S)=U(p)$ by internality to a minimal set, Conjecture~\ref{BCboundConj} implies that $d-1\leq n+2$.
This proves part~(a).

For part~(b), we need only rule out the extreme case when $d=n+3$.
But in that case $(G,S)$ is a generically $(n+2)$-transitive definable homogeneous space.
So, Conjecture~\ref{BCConj} applies and we have that $(G,S)$ is (abstractly) isomorphic to $\big(\PGL_{n+1}(F),\PP^n(F)\big)$ for some algebraically closed field~$F$.
Fix distinct elements $b,c\in S$, let $H:=\stab(b,c)\leq G$, and consider the action of $H$ on $S$.
Then, besides $\{b\}$ and $\{c\}$, $H$ has exactly two orbits in $S$, both of which are infinite.
Indeed, this is the case for $\big(\PGL_{n+1}(F),\PP^n(F)\big)$ where one of the orbits is the line on which $b$ and $c$ lie and the other is the complement of that line.
As $S$ is of Morley rank $n$ and Morley degree one, one of these infinite orbits, say $\mathcal O$, must be of rank $<n$.
Write $\mathcal O$ as $Hd$ for some $d\in S$, and let $\widehat q$ be the generic type of $\mathcal O$ over $Abcd$.
Then $\widehat q$ is a nonalgebraic forking extension of $q$.
It follows, by definition, that $\nmdeg(q)\leq 3$.
But this contradicts $\nmdeg(q)=d=n+3>4$.
Hence, $d\neq n+3$ and we have $d\leq n+2$, as desired.
\end{proof}

\begin{corollary}
Under Assumption~\ref{nlm}, if $p$ is stationary and of $U$-rank $2$ then $\nmdeg(p)\leq 4$.
\end{corollary}

\begin{proof}
In proving Theorem~\ref{bound}(b) we only applied Conjecture~\ref{BCConj} to a group action $(G,S)$ where $\RM(S)=U(p)$.
But, when $\RM(S)=2$, Conjecture~\ref{BCConj} is a theorem of Alt\i nel and Wiscons~\cite{altinel2018recognizing}.
\end{proof}

Actually we get more from the proof of Theorem~\ref{bound}.
We only applied the Borovik-Cherlin conjectures to group actions arising from internality to a nonmodular minimal type-definable set.
So one would expect that in cases where one understands well the nonmodular minimal types one could prove the conjectures outright for the relavant group actions and thus deduce the bound on degree of nonminimality.
And indeed, this is what we do in the next sections; we prove that Conjecture~\ref{BCConj} holds for algebraic groups in characteristic $0$ and use that to deduce the bound $\nmdeg(p)\leq U(p)+2$ in differentially closed fields and compact complex manifolds.
With a bit more work, we will get that down to $U(p)+1$ in these theories.

\begin{remark}
It is worth pointing out that the existence, unconditionally and in general, of {\em some} bound on $\nmdeg$ in terms of $U$-rank can be deduced from the above methods.
Indeed, it is shown in~\cite[Corollary~2.1 and Lemma~1.20]{BC2008} that there is a function $\tau:\omega\to\omega$ such that whenever $(G,S)$ is a transitive finite rank group action that is generically $k$-transitive then $k\leq\tau(\RM(S))$.
(Conjecture~\ref{BCboundConj} is that $\tau(n)=n+2$ works.)
It follows, under Assumption~\ref{nlm} and using the above methods, that $\nmdeg(p)\leq\tau(U(p))+1$. 
\end{remark}

\bigskip
\section{The Borovik-Cherlin conjecture in $\ACF_0$}
\label{section-bcalg}
\noindent
As has been suggested in various places~\cite{BC2008, altinel2018recognizing, borovik2019binding}, one should be able to prove the conjectures of Borovik and Cherlin for algebraic groups in characteristic zero using the work of Popov~\cite{Popov2007}  to deal with simple linear algebraic groups and then using the O'Nan-Scott Theorem of Macpherson and Pillay~\cite{macpherson1995primitive} to reduce to the simple case.
As we have not seen this carried out, we will do so here in some detail.

We work throughout this section in a sufficiently saturated $\U\models\ACF_0$.

First, we verify that that the finite Morely rank notion of generic transitivity given in~Definition~\ref{genBC} agrees with that introduced by Popov~\cite{Popov2007} for algebraic groups.

\begin{lemma}
Suppose $G$ is an algebraic group and $\alpha:G\times S\to S$ is an algebraic action of $G$ on an irreducible variety $S$.
Then, for all $n\geq 1$, the action $\alpha$ is generically $n$-transitive in the sense of definition~\ref{genBC} if and only if the induced action of $G$ on $S^n$ has a Zariski open orbit.
\end{lemma}

\begin{proof}
Since Morley rank agrees with dimension in $\ACF_0$, and a proper subvariety of an irreducible variety is of strictly smaller dimension, the existence of a Zariski open orbit for the action of $G$ on $S^n$ does imply generic $n$-transitivity of $(G,S)$.
For the converse, assume $\mc O\subseteq S^n$ is an orbit of $(G,S^n)$ whose complement is of strictly smaller dimension.
In particular, $\mc O$ is Zariski dense in $S^n$. 
We show that $\mc O$ is Zariski open.
Being a definable Zariski dense set, it does contain a nonempty Zariski open subset, $U$.
Fix $u\in U$.
For any $a\in\mc O$ there is $g\in G$ such that $\alpha(g,u)=a$.
Now, $f:=\alpha(g,-):S^n\to S^n$ is an algebraic automorphism that preserves $\mc O$, and hence $f(U)$ is Zariski open, contains $a$, and is contained in $\mc O$.
We have shown that every element of $\mc O$ has a Zariski open neighbournhood contained in $\mc O$.
This implies that $\mc O$ is Zariski open, as desired.
\end{proof}

\bigskip
\subsection{The case of simple linear algebraic groups}
We explain how Conjecture~\ref{BCConj} for simple linear algebraic groups in charactertistic zero can be deduced from the statements and arguments appearing in Popov~\cite{Popov2007}.

\begin{theorem}
\label{BCsimplealg}
Suppose $\alpha:G\times S\to S$ is an algebraic action of a simple linear algebraic group on a positive-dimensional irreducible variety in characteristic zero.
If $\alpha$ is transitive and generically $(n+2)$-transitive then $\alpha$ is isomorphic to the natural action of $\PGL_{n+1}$  on $\m P ^n$.
\end{theorem}

\begin{proof}
As in~\cite{Popov2007}, by the {\em generic transitivity degree of $\alpha$}, denoted by $\gtd(\alpha)$ or $\gtd(G,S)$, we mean the supremum among all $d\geq0$ such that $\alpha$ is generically $d$-transitive.
We are given that $\gtd(\alpha)\geq n+2$.

We proceed by induction on $n$.
For the $n=1$ case we use the fact that the faithful linear algebraic group actions on curves are completely classified as the natural action of the additive or multiplicative group on the affine line, the semidirect product of the two acting by affine transformations, or $\PGL_2$ acting on $\mathbb P$ by projective transformations.
Of these, only the last is generically $3$-transitive.
So we may assume that $n>1$.

Let $k$ be an algebraically closed subfield over which $\alpha$ is defined.
Fix $x_0\in S(k)$ and let $H:=\operatorname{Stab}(x_0)$.
This is a proper closed subgroup of $G$ over $k$.
As $\alpha$ is transitive, we have a $G$-equivariant isomorphism between $G/H$ and $S$ given by $gH\mapsto \alpha(g,x_0)$.
So we may as well assume that $S=G/H$ and $\alpha$ is the natural action of $G$ on $G/H$. 

Theorem~1 of~\cite{Popov2007}  gives upper bounds on the generic transitive degree of the possible actions of a linear algebraic group depending on its classification-type.
In particular, there are only two types of simple linear algebraic groups that admit actions of generic transitive degree greater than $3$.
As $\gtd(\alpha)\geq n+2>3$, this applies to $\alpha$, and the two possibilities are:
\begin{itemize}
\item[(i)]
$G$ is of type $E_6$, and $\gtd(\alpha)=4$, or
\item[(ii)]
$G$ is of type $A_\ell$ and $\gtd(\alpha)\leq \ell+2$.
\end{itemize}
We first rule out case~(i).
Indeed, if $G$ is of type $E_6$ then it is of dimension $78$ and one has an explicit list of the classification-types (and hence dimensions) of the maximal proper closed subgroups of $G$ -- see, for example, Table~1 of~\cite{MR2044850}.
In particular, one sees that $\dim(H)\leq 22$, and hence $n=\dim(G/H)>3$.
But then $\gtd(\alpha)=4$ contradicts $\gtd(\alpha)\geq n+2$.

Therefore $G$ is of type $A_\ell$, and so, being simple, $G=\PGL_{\ell+1}$.
Our strategy now is to show that $H$ must be a maximal standard parabolic\footnote{A closed subgroup is {\em parabolic} if the quotient is projective and is {\em standard} if it contains the subgroup of upper triangular matrices in $G$.} subgroup $P\leq\PGL_{\ell+1}$, in which case the action of $\PGL_{\ell+1}$ on the cosets of $P$ agrees with its action on a grassmanian, and one can deduce that the only way that action can be of generic transitivity degree $\geq n+2$ is if $n=\ell$ and $\PGL_{n+1}/P=\mathbb P^n$.

Let $H'$ be a maximal proper closed subgroup of $G$ containing $H$.
The projection $G/H\to G/H'$ is $G$-equivariant, and hence, by Lemma~2(i) of~\cite{Popov2007},
$$\gtd(G,G/H')\geq\gtd(G,G/H)\geq n+2.$$
If $\dim(G/H')=d<n$ then we get $\gtd(G,G/H')>d+2$, which contradicts the inductive hypothesis that $(G,G/H')$ is isomorphic to $(\PGL_{d+1},\mathbb P^d)$.
So $d=n$ and $H$ is of finite index in $H'$.
It follows that $G/H\to G/H'$ is a finite \'etale cover.
Hence, if we prove that $G/H'=\PP^n$ then it would follow that $H=H'$.
So it suffices to prove the result for $(G,G/H')$.
That is, we may assume $H=H'$ is a maximal proper closed subgroup of $G$.

It is a fact that for maximal proper closed subgroups, either the connected component is reductive or the subgroup is parabolic -- see~\cite[$\S30.4$]{humphreys}.
Let $H^\circ$ be the connected component of $H$.
By Lemma~2(ii) of~\cite{Popov2007}, we have that $\gtd(G:G/H)=\gtd(G:G/H^\circ)$.
Now, if $H^\circ$ were reductive, then by Lemma~6 of~\cite{Popov2007} we would have $\gtd(G:G/H^\circ)=1$, contradicting that it is $\geq n+2$.
Hence $H^\circ$ is not reductive and $H$ is parabolic.
(Note that this implies $H=H^\circ$, see~\cite[$\S23.1$~Corollary~B]{humphreys}.)

Every parabolic subgroup is conjugate to a standard one (\cite[$\S29.3$]{humphreys}).
As conjugation induces a $G$-equivariant algebraic isomorphism, we may assume that $H=P$ is a standard parabolic subgroup.
Now, it is well known that a standard parabolic subgroup $P\leq \PGL_{\ell+1}$ is the stabiliser of a (partial) flag in $\mathbb A^{\ell+1}$ under the natural action of $\GL_{\ell+1}$, and so the homogeneous space $\PGL_{\ell+1}/P$ is a flag variety.
When $P$ is maximal (as $H$ is), the flag is of length $1$; it is just a dimension $m>0$ proper subspace of $\mathbb A^{\ell+1}$.
Hence, in that case, $\PGL_{\ell+1}/P$ is the grassmannian $\GR(m,\ell+1)$.
So
$(G,G/H)=\big(\PGL_{\ell+1},\GR(m,\ell+1)\big)$.
In particular, $n=\dim\GR(m,\ell+1)=m(\ell+1-m)$.
The generic transitivity degree of $\big(\PGL_{\ell+1},\GR(m,\ell+1)\big)$ is computed in Theorem~3 of~\cite{Popov2007} to be the greatest integer bounded above by $\displaystyle \frac{(\ell+1)^2}{m(\ell+1-m)}$.
As $\gtd(G,G/H)\geq n+2$, we get that
$\displaystyle \left\lfloor\frac{(\ell+1)^2}{m(\ell+1-m)}\right\rfloor\geq m(\ell+1-m)+2$.
It is easily checked that this forces $m=1$.
So $n=\ell$ and the action of $G$ on $G/H$ is that of $\PGL_{n+1}$ on $\GR(1,n+1)=\PP^n$.
\end{proof}

\bigskip
\subsection{The general case}
\label{reducetosimple}
We now establish Conjecture~\ref{BCConj} for $\ACF_0$.

\begin{theorem}
\label{BCacf}
Suppose, in $\ACF_0$, we have a definable connected homogeneous space $(G,S)$ with $S$ of dimension $n>0$.
If $G$ acts generically $(n+2)$-transitively on $S$ then $(G,S)$ is definably isomorphic to $\big(\PGL_{n+1},\PP^n\big)$.
\end{theorem}

\begin{proof}
First of all, as a consequence of the Weil group-chunk (or rather homogeneous-space-chunk) theorem~\cite{weil}, we know that $(G,S)$ is definably isomorphic to an algebraic group action.
That is, we may assume $\alpha:G\times S\to S$ is an algebraic action of an algebraic group on an irreducible variety -- which is faithful, transitive, and generically $(n+2)$-transitive -- and we aim to prove that $\alpha$ is isomorphic (as an algebraic group action) to the natural action of $\PGL_{n+1}$ on $\PP^n$.

Again we proceed by induction on $n$.
When $n=1$ we only have, besides the linear algebraic group actions on curves (of which, as mentioned in the previous section, only the action of $\PGL_2$ on $\mathbb P$ is generically $3$-transitive), the action of an elliptic curve on itself (which, being regular, has generic transitive degree $1$).
So we may assume that $n>1$.

Next, we reduce to the {\em definably primitive} case, meaning that there is no nontrivial definable equivalence relation $E$ on $S$ that is $G$-equivariant in the sense that $E(x,y)\iff E(gx,gy)$ for all $g\in G$.
This reduction is actually effected in general for finite Morley rank groups by Borovik and Cherlin in~\cite{BC2008} (see the bottom of page~35), but the situation is much simpler in the case of algebraic groups.
Indeed, exactly as in the proof of Theorem~\ref{BCsimplealg}, transitivity along with the induction hypothesis and the simple-connectedness of $\PP^n$, allows us to reduce to the case that $(G,S)=(G,G/H)$ where $H$ is a maximal proper closed subgroup of $G$.
As all definable subgroups of $G$ are closed, we have that $H$ is a maximal proper definable subgroup.
But this is equivalent to $(G,G/H)$ being definably primitive.

So we may assume that $(G,S)$ is definably primitive.
This puts us into the context of the O'Nan-Scott theorem of Macpherson and Pillay, namely~\cite[Theorem~1.1]{macpherson1995primitive}, which we now use to reduce to $G$ being a simple linear algebraic group.

Let $B$ be the {\em definable socle} of $G$, that is, the subgroup generated by the minimal normal definable subgroups of $G$.
Then, by standard finite Morley rank techniques, $B$ is itself definable.
The O'Nan-Scott theorem gives a list of possibilities, labelled as 1, 2, 3, 4a(i), 4a(ii), 4b in~\cite[Theorem~1.1]{macpherson1995primitive},  for the structure of $B$ and how close $G$ is to $B$.
Before dealing with the various cases individually (and even stating what they are) let us point out that,  as explained right after the statement of Theorem~1.1 in~\cite{macpherson1995primitive}, the fact that~$G$ is connected already rules out cases~4a(i) and~4b, and so we will not discuss these.

In case~1 of~\cite[Theorem~1.1]{macpherson1995primitive}, $B$ is either torsion-free divisible or an elementary abelian $p$-group.
The known structure of commutative algebraic groups in characteristic zero rules out the latter (see for example Theorem~5.3.1 of~\cite{brion}).
In the torsion-free divisible case we have Lemma~3.2 of~\cite{BC2008} which tells us that $\RM(G)\leq n^2+n$.
But generic $(n+2)$-transitivity implies that $\RM(G)\geq (n+2)n$.
This contradiction rules out case~1.

We are left with three cases, namely 2, 3, and 4a(ii) of~\cite[Theorem~1.1]{macpherson1995primitive}, which we now finally state:
\begin{itemize}
\item[2.]
$B$ is noncommutative simple, acts regularly on $S$, and $G$ is an extension of $B$ by a subgroup $H\leq \aut B$ such that $H\cap B=(e)$, or
\item[3.]
$B$ is noncommutative simple and $B\leq G\leq\aut B$, or
\item[4a(ii).]
$B=T_1\times T_2$ where $T_1, T_2$ are (definably isomorphic) infinite simple noncommutative definable normal subgroups of $G$ both acting regularly on $S$, and $B\leq G\leq W$ where $W$ is an extension of $B$ by $(\aut T_1/T_1)\times\sym_2$.
\end{itemize}
Here, for a simple group $T$ we view $T\leq \aut T$ with elements acting by conjugation.
Of course, the actual Theorem~1.1 in~\cite{macpherson1995primitive} gives more information in each of these cases; we have only recorded what we require.

We first argue that, in all three cases, $G=B$.
This is because for a simple linear algebraic group $T$, the outer automorphism group $\aut T/T$ is finite (see~\cite[$\S27.4$]{humphreys}).
Note that $B$ in cases~2 and~3, and $T_1, T_2$ in case~4a(ii), are simple linear algebraic groups as they are simple noncommutative definable subgroups of the algebraic group $G$.
So, inspecting the three cases, we see that in each of them $G$ is a finite extension of $B$.
But then connectedness forces $G=B$.

Case~2 is therefore impossible as $G$ acting regularly on $S$ is incompatible with it acting generically $(n+2)$-transitively.
Case~4a(ii) is also impossible for similar reasons: as $T_1$ and $T_2$  act regularly on $S$, $G=T_1\times T_2$ would imply that $\dim G=2n$, again contradicting generic $(n+2)$-transitivity (which implies $\dim G\geq (n+2)n$).

So we are in case~3 and $G=B$ is itself a simple linear algebraic group.
Theorem~\ref{BCsimplealg} applies, and $(G,S)=\big(\PGL_{n+1},\PP^n\big)$, as desired.
\end{proof}

\bigskip
\section{Differentially closed fields}
\label{section-dcf}
\noindent
Equipped with the truth of Conjecture~\ref{BCConj} in~$\ACF_0$, we can now established the desired upper bound on degree of minimality in theories where $\ACF_0$ is the site of all nonmodular minimal types.
For concreteness we here only consider $\DCF_{0}$, the theory of differentially closed fields of characteristic zero.

We work in a saturated model $\U\models \DCF_{0}$ with field of constants $\C$.

We begin by improving Proposition~\ref{wo} to a strict inequality:

\begin{proposition}
\label{wodcf}
Suppose $p\in S(A)$ is stationary nonalgebraic type of finite rank that is $\C$-internal and weakly $\C$-orthogonal.
Then
$$\nmdeg(p)< \min\{k:p^{(k)}\text{ is not weakly $\C$-orthogonal}\}.$$
\end{proposition}

\begin{proof}
Since $p$ is weakly $\C$-orthogonal and $\C$-internal, it is isolated.
So $S:=p(\U)$ is an $A$-definable set.
Since $p$ is nonalgebraic and $\C$-internal, there is $k>0$ such that $p^{(k)}$ is not weakly $\C$-orthogonal.
Let $k$ be minimal such, fix $(a_1,\dots,a_k)\models p^{(k)}$, and let $B:=Aa_1,\dots,a_{k-1}$.
We have seen, in the proof of Proposition~\ref{wo}, that $\tp(a_k/B)$ admits a proper fibration $\tp(b/B)$ where $b\in\C$.
Write $b=f(a_k)$  where $f:X\to Y$ is a $B$-definable surjective function with $a_k\in X\subseteq S$ and $b\in Y\subseteq\C$.

Since $a_k\notin\acl(Bb)$, the fibre $f^{-1}(b)$ is not finite.
By elimimination of the infinity quantifier, we may, by shrinking $Y$ if necessary, assume that none of the fibres of $f$ are finite.
On the other hand, since by stationarity at most one of the fibres can have Morley rank equal to that of $S$, and $f$ has infinitely many fibres as $b\notin\acl(B)$, we may shrink $Y$ further if necessary so that all the fibres are of Morley rank strictly less than~$\RM(S)$.
It follows that for every $c\in Y$ there is $a\in f^{-1}(c)$ such that $\tp(a/Bc)$ is a nonalgebraic forking extension of $\tp(a/A)=p$.

Now, by stable embeddability, $Y$ is $(B\cap\C)$-definable in the pure field $(\C,+,\times)$, and hence must have an $\acl(B\cap\C)$-point, say $c$.
Let $a\in f^{-1}(c)$ be such that $\tp(a/Bc)$ is a nonalgebraic forking extension of $p$.
As $c\in\acl(Aa_1,\dots,a_{k-1})$, we have that $\tp(a/Aa_1,\dots,a_{k-1})$ is a nonalgebraic forking extension of $p$.
That is, $\nmdeg(p)\leq k-1$.
\end{proof}

\begin{theorem}
\label{bound-dcf}
For every stationary finite rank type $p\in S(A)$ in $\DCF_{0}$,
$$\nmdeg(p)\leq U(p)+1.$$
\end{theorem}

\begin{proof}
The general strategy is that of Theorem~\ref{bound}(b) but using the fact that the Borovik-Cherlin Conjecture holds of homogeneous spaces defined in the constants (by Theorem~\ref{BCacf}).
We get a bound that is one less than that of Theorem~\ref{bound}(b) because we use the strict inequality of Proposition~\ref{wodcf} in place of the non-strict one of Proposition~\ref{wo}.
But here are some details.

Let $d:=\nmdeg(p)$.
We may as well assume that $d>1$.
It follows by Proposition~\ref{internality} that $p$ is almost internal to some nonmodular minimal type, and hence by the Zilber trichotomy in $\DCF_{0}$, we have that $p$ is almost $\C$-internal.
By Lemma~\ref{basicprop}(c), interalgebraicity does not change the degree of nonmininimality, so we may assume that $p$ is $\C$-internal.
On the other hand, by Proposition~\ref{wo} and the fact that $d>1$, we must have that $p$ is weakly $\C$-orthogonal.
So Proposition~\ref{wodcf} applies.
Following the proof of Proposition~\ref{gtbound}, but using~\ref{wodcf} instead of~\ref{wo}, we see that if we let $G=\aut_A(p/\C)^\circ$ be the connected component of the binding group, and $S:=p(\U)$, then $(G,S)$ is a definable homogeneous space that is generically $d$-transitive.
The improvement here over Proposition~\ref{gtbound} is that we have obtained one higher level of generic transitivity.

Now, $(G,S)$ is definably isomorphic (over possibly additional parameters) to some $(\widehat G,\widehat S)$ definable in $(\C,+,\times)$.
This is because $(\C,+,\times)$ is purely stably embedded.
So $(\widehat G,\widehat S)$ is a connected homogeneous space definable in $\ACF_0$ with $\dim\widehat S=U(p)=:n>0$
that is generically $d$-transitive.

Suppose toward a contradiction that $d\geq n+2$.
So, by Theorem~\ref{BCacf},  $(\widehat G,\widehat S)$, and hence $(G,S)$, is definably isomorphic to $\big(\PGL_{n+1}(\C),\PP^n(\C)\big)$.
But, as in the proof of Theorem~\ref{bound}(b), this implies that $d\leq 3$ which would force $n\leq 1$ and hence $d=0$ by convention.
This contradiction proves that $d\leq n+1$, as desired.
\end{proof}

Note that while we obtain the desired bounds on degree of nonminimality in $\DCF_{0}$, we do not resolve the Borovik-Cherlin conjectures in this theory.

\bigskip
\section{Compact complex manifolds}
\label{section-bcccm}
\noindent
The arguments we gave to prove the bound on nonminimality in $\DCF_{0}$ work in any theory where the Zilber trichotomy holds; namely, where there is a stably embedded pure algebraically closed field of characteristic zero to which every nonmodular minimal type is nonorthogonal.
So we obtain the analogue of Theorem~\ref{bound-dcf} for the theory of compact complex manifolds:

\begin{theorem}
\label{bound-ccm}
For every stationary finite rank type $p\in S(A)$ in $\ccm$,
$$\nmdeg(p)\leq U(p)+1.$$
\end{theorem}

\noindent
We leave it to the reader to check the details.
But it turns out that, unlike in $\DCF_{0}$, we can prove the Borovik-Cherlin conjecture itself in $\ccm$.
Indeed, the proof for $\ACF_0$ that we gave in Section~\ref{section-bcalg} works more generally for groups definable in $\ccm$.
We give a few brief explanations here.

First of all, $\ccm$ is the theory of the multi-sorted structure $\mathcal A$ where there is a sort for each compact complex analytic space and where the language consists of a predicate for each closed analytic subset of a finite cartesian product of sorts.
The theory has many nice properties: it admits the elimination of quantifiers and imaginaries, each sort is of finite Morley rank, and a Zilber trichotomy holds of the strongly minimal sets. 
It is an expansion of $\ACF_0$ in the sense that the complex field $(\mathbb C,+,\times)$ is purely stably embedded in the sort of the projective line.
The standard model is not saturated, and some complications arise from the fact that one has to work in a sufficiently saturated elementary extension $\mathcal A'\succeq\mathcal A$ whose sorts are not complex analytic spaces.
Each sort is, however, endowed with a noetherian {\em Zariski topology} whose closed sets come from definable families of closed anlaytic sets in the standard model.
See~\cite{ccs-survey} for a detailed survey of the subject.

The groups definable in $\ccm$ are well understood -- they are precisely the {\em meromorphic groups} studied in~\cite{fujiki-mergroups, pillay-scanlon-mergroups, ams-mergroups, scanlon-mergroups}.
In the standard model they are a natural generalisation of algebraic groups to the meromorphic category: complex Lie groups with a finite covering by Zariski open subsets of irreducible compact complex spaces such that the transition maps and the group operation extend to meromorphic maps.
This notion was extended to the saturated universe~$\mathcal A'$ in~\cite[Definition~4.3]{ams-mergroups}, which we leave the reader to consult for a precise definition.
In any case, every group definable in $\mathcal A'$ admits the structure of a meromorphic group.
Coming out of~\cite{fujiki-mergroups, pillay-scanlon-mergroups, ams-mergroups, scanlon-mergroups} is a Chevellay-type structure theorem for meromorphic groups that allows the arguments of the previous section to extend from $\ACF_0$ to $\ccm$.

\begin{theorem}
\label{BCccm}
Conjecture~\ref{BCConj} holds of homogeneous spaces definable in $\ccm$.
\end{theorem}

\begin{proof}
Inspecting the proof of Theorem~\ref{BCacf} we see that all we require is the following three facts about meromorphic groups:
\begin{enumerate}
\item
all meromorphic homogeneous spaces $(G,S)$, with $\dim S=1$, are algebraic,
\item
no infinite meromorphic group is an elementary abelian $p$-group, and
\item
every simple noncommutative meromorphic group is a simple linear algebraic group.
\end{enumerate}
Indeed, given these, the reduction in~$\S$\ref{reducetosimple} to simple linear algebraic groups goes through verbatum.

So it remains to verify the above properties of meromorphic groups.

In $\mathcal A'$ every $1$-dimensional set is algebraic; this is the nonstandard Riemann existence theorem of~\cite{ret}.
It is also shown there that $\mathbb C(\mathcal A')$ is the only infinite field definable in $\mathcal A'$.
Property~(1) follows from these facts together with the classification of strongly minimal homogeneous spaces (see~\cite[1.6.25]{GST}).

In $\mathcal A'$ every meromorphic group arises as a nonstandard fibre of a definable family of meromorphic groups in the standard model.
So if there did exist an infinite meromorphic group that is an elementary abelian  $p$-group, then there would exist a standard one.
By the Chevellay-type structure theorem of~\cite{pillay-scanlon-mergroups}, every standard commutative meromorphic group is an extension of a complex torus by a commutative complex linear algebraic group, and hence is not an elementary abelian $p$-group.
This verifies~(2).

Finally, for~(3), we apply the nonstandard Chevellay-type structure theorem for meromorphic groups in $\mathcal A'$ established in~\cite{ams-mergroups, scanlon-mergroups}.
This says that every meromorphic group is the extension of a (commutative) definably compact\footnote{We mean this as a group interpretable in $\mathbb R_{\operatorname{an}}$, but in fact one can say more, see the discussion in the Introduction to~\cite{ams-mergroups}.}
meromorphic group by a linear algebraic group.
In particular, if $G$ is a simple meromorphic group then it is either linear algebraic or commutative.
\end{proof}

%\bibliography{Research}{}
%\bibliographystyle{plain}

\end{document}